\documentclass{amsart}

\usepackage{amssymb}
\usepackage{amsmath}
\usepackage{color}
\usepackage{graphicx}
\def\E{{\mathbb E}}
\def\P{{\mathbb P}}
\def\R{{\mathbb R}}
\def\Z{{\mathbb Z}}

\def\N{{\mathbb N}}

\def\cS{{\mathcal S}}

\def\S{ {\mathcal S}}

\def\PZ{{\mathcal P}}

\newtheorem{theo}{Theorem}
\newtheorem{prop}{\indent Proposition}

\newtheorem{rem}{\indent Remark}

\title[Information transmission and criticality]{Information transmission and criticality in the contact process.}

\date{October 30, 2017}
\author{M. Cassandro, A. Galves, E. L\"ocherbach}

\address{M. Cassandro :  Gran Sasso Science Institute, 
L'Aquila,
Italy.}

\email{marzio.cassandro@gmail.com}

\address{A. Galves: Universidade de S\~ao Paulo, Instituto de Matem\'atica e Estat\'istica, 
S\~ao Paulo, Brazil}

\email{galves@usp.br}

\address{E. L\"ocherbach: Universit\'e de Cergy-Pontoise, 
AGM, CNRS-UMR 8088,
95000 Cergy-Pontoise,
France.}

\email{eva.loecherbach@u-cergy.fr}

\subjclass[2010]{60K35; 82B27}

\keywords{Contact process. Criticality. Information transmission. Duality and coupling.}

\begin{document}

\maketitle

\begin{abstract}
In the present paper, we study the relation between criticality and information transmission in the one-dimensional contact process with infection parameter $\lambda .$ We introduce a notion of {\it sensitivity} of the process to its initial condition and prove that it  increases not only for values of $\lambda < \lambda_c, $ the value of the critical parameter, but keeps increasing even after  $ \lambda_c  , $ before finally starting to decrease for  values of $\lambda  $ sufficiently above $\lambda_c.$ This provides a counterexample to the common belief that associates maximal information transmission to criticality.

\end{abstract}

\section{Introduction}
From swarms of birds to neuronal activity, an increasing number of recent papers claims that the ability of a complex system to transmit information is maximized at the {\it critical} point. Just to cite a few examples, \cite{grogo} advertise that in swarms of cooperative units such as bird flocks ``the information transfer is made possible by the nonlocal nature of
the criticality condition.'' See also \cite{Cavagna29062010} who claim that ``flocks behave as critical systems, poised to respond maximally to environmental perturbations.'' Similar ideas have emerged in neurobiology. 

It is reasonable to conjecture that these ideas originate under the influential work of Per Bak, see for instance  \cite{bak} where he devotes a whole chapter to the question {\it ``Why Should the Brain Be Critical? ''} \cite{BeggsPlenz} suggest the following answer.  {\it ``The fact that the critical state [...] maximized information transmission in these networks is consistent with an intuitive understanding of how a branching process would work in the context of a highly parallel network.
If the network were subcritical, an input signal would attenuate, causing most output units to be inactive, thus leaving little evidence of the input. If the network were supercritical, any input signal would eventually lead to most output units being active, again leaving little information as to what the input was.''} This interpretation echoes Per Bak's own answer given in \cite{bak}: {\it ``The brain must operate at the critical case where the information is just able to propagate''.} 

A preliminary step in this direction is to clarify what we mean by {\it information transmission} in a complex system and how to measure it. From a neuroscientific point of view, \cite{KinouchiCopelli} suggested the following answer to this question. They propose to use the notion of {\it dynamical range} borrowed from acoustics to measure the sensitivity of the system to external stimuli. More precisely, they consider a stochastic system of interacting neurons exposed to an external stimulus modeled by a Poisson point process. In their model, the graph of interactions is given by an undirected Erd\"os-R\'enyi random graph. For this model, the authors are able to precisely define the notion of criticality, using as relevant parameter the average branching ratio. In fact they show, by numerical simulations, that a critical parameter value exists such that the dynamical range increases monotonically below this parameter and decreases monotonically above. 

In the present paper, we formulate the problem of information transmission in the following way.  We study to which extent a system is able to discriminate between two different initial stimuli. We do this for the one-dimensional contact process which is probably the simplest non-trivial complex system one might think of. More precisely, we compare two coupled time evolutions starting from Bernoulli product measures on a finite set of points having different densities. We then define the {\it sensitivity} of the model with respect to the initial signal as the total variation distance between the two associated processes at a given single site, at some fixed time.  We  prove that the sensitivity of the system to the initial condition, as a function of $\lambda ,$ keeps increasing even after the critical point, before finally starting to decrease. This contradicts the common belief  that information transmission  is maximized at the critical point.

In Section \ref{sec:finaldiscussion}, we discuss our results in comparison with those, already present in the literature, where similar features are pointed out  by numerical simulations.

This paper is organized as follows. In Section \ref{sec:def}, we recall the definition of the one-dimensional contact process, we give the definition of our measure of sensitivity in \eqref{eq:sensitivitydiff} and state Theorem \ref{theo:1} which is our main result. The proofs are collected in Section \ref{sec:proof}. 

\section{Definitions and main results}\label{sec:def}
In the following, we briefly recall the definition of the one-dimensional contact process introduced in \cite{harris1974}. 
Let $X =\{ 0, 1\}^\Z $ and write $ \xi = (\xi (i) )_{i \in \Z } $ for the elements of $ X.$ The contact process on $\Z $  is the continuous time Markov process $ (\eta^\lambda_t (i) ,  i \in \Z , t \geq 0) ,$ 
taking values in $X$ having  generator  
\begin{equation}\label{eq:additive} 
L f (\xi) = \sum_{ i \in \Z}c(i, \xi) [  f ( \xi^i ) - f (\xi ) ]  ,
\end{equation}
for any cylinder function $f .$ In the above formula, $\xi^i $ the configuration defined by
 \begin{eqnarray*}
 \xi^i   (j)  &= &\xi (j)  , \mbox{ for all $ j \neq i ,$ }\\
\xi^i  (i)  &= & 1- \xi(i ) ,
\end{eqnarray*}
and  
$$
c(i, \xi) = \left\{ 
\begin{array}{ll}
1 & \mbox{ if } \xi ( i ) = 1 \\
\lambda \sum_{j = i \pm 1} \xi (j) & \mbox{ if } \xi ( i ) = 0 
\end{array}  \right\} .
$$
Here, $\lambda > 0 $ is a fixed constant. We shall write $ (\eta^{\lambda, \xi}_t , t \geq 0 ) $ for a version of the above process starting from $ \eta^{\lambda , \xi}_0 = \xi, $ for any fixed initial configuration $ \xi \in X.$

Recall that \cite{harris1974} proved that there exists a critical value $ \lambda_c$ with $ 0 < \lambda_c < \infty $ such that for all $\lambda <  \lambda_c, $ there is only one invariant measure, the Dirac-measure supported by the $0-$configuration, while for $ \lambda > \lambda_c, $ a second and non-trivial extremal invariant measure appears. This result was completed by \cite{Bez} who prove that at the critical point only one invariant measure, the trivial one, exists. In particular,  in the subcritical case,  starting from any initial configuration, the process converges to the zero-configuration, while in the supercritical case it converges to a convex combination of the two extremal invariant measures. In the supercritical case, the influence of the initial configuration appears in the weighting factor defining this mixture. 

Here, to measure the sensitivity to the initial condition we must adopt a non-asymptotic point of view and measure, at a given time $t,$ how well the system discriminates between two different initial states. The precise definition is as follows. We suppose that the system  
starts with a initial configuration with a given density $p$ within a finite subset $ \Lambda $ of $\Z .$ This process will be denoted $ (\eta_t^{\lambda, p, \Lambda }, t \geq 0 ) ,$ where $ \lambda $ is the intensity of the infection rate appearing in \eqref{eq:additive}. We suppose that $ \eta^{\lambda, p, \Lambda }_0 (  i ) = \eta (i )    $ for all $ i \in  \Lambda  , $ where $ \eta (i) ,  i \in \Lambda , $ are i.i.d.\ Bernoulli random variables with parameter $p.$ We also suppose that $ \eta^{\lambda, p ,\Lambda   }_0  (i) = 0 $ for all $ i \in \Lambda^c .$ 

We consider two intensities $ p < q $ and the associated processes $ (\eta_t^{\lambda, p, \Lambda } , t \geq 0 )  $ and $(\eta_t^{\lambda, q, \Lambda } , t \geq 0)  . $ We then define the {\it sensitivity} of the process with parameter $\lambda $ with respect to the initial condition by
\begin{equation}\label{eq:dynrange}
\S_{p,q, \Lambda} ( \lambda   , t ) :=    \inf_{ \P \in \Pi   } \E ( | \eta_t^{\lambda, p, \Lambda} (0)) -  \eta_t^{\lambda, q, \Lambda} ( 0) )| ) ,
\end{equation}
where $ t > 0 $ and where the infimum is taken over all possible couplings $\P $ of $\eta_t^{\lambda, p, \Lambda }(0) $ and $ \eta_t^{\lambda, q, \Lambda}  (0) .$ 
The quantity $\S_{p, q, \Lambda} ( \lambda , t) $ measures the minimal distance of the two processes in a given position (here, the position $0$) at time $t, $ under two initial Bernoulli configurations of density $ p $ and $q.$ 

In the following we choose $ \Lambda := \Lambda_r = \{ - r , r \} , $ for a fixed position $r \in \Z,$ and we pose for any $0 <  \lambda_1 < \lambda_2 < \infty    , $    
\begin{equation}\label{eq:sensitivitydiff}
\Delta_{p,q} (\lambda_1, \lambda_2, r, t  ) =   \S_{p,q, \Lambda_r} ( \lambda_2, t     ) - \S_{p,q, \Lambda_r} ( \lambda_1   , t  ) .
\end{equation} 
The quantity $\Delta_{p,q} (\lambda_1, \lambda_2, r, t  )$ measures the sensitivity variation with respect to increasing values of the intensity $ \lambda ,$ at time $t.$ We show that this sensitivity variation is non-decreasing for all  $\lambda < \lambda_c  $ and continues increasing even after $\lambda_c $ before finally being decreasing. This is the content of our main theorem that we present now.

\begin{theo}\label{theo:1}
For any fixed $q > p > \frac23  $  there exist $ \lambda_c < \lambda_1 (p ) \le  \lambda_2 ( p ) $ such that the following holds.

For $\lambda_1 < \lambda_2 < \lambda_1 (p ) , $  there exist $r^* = r^* (\lambda_1 , p, q) $ and  $ t^*= t^* (\lambda_1 , p, q, r^*) , $  such that for all $ t \geq t^*  $ and $r \geq r^* , $ 
$$ \Delta_{p,q} (\lambda_1, \lambda_2 ,r, t  ) > 0  .$$

For $ \lambda_2 > \lambda_1 > \lambda_2 (p ) , $ there exist $r^* = r^* (\lambda_1 , p, q) $ and  $ t^*= t^* (\lambda_1 , p, q, r^*) , $  such that for all $ t \geq t^*  $ and $r \geq r^* , $  
$$  \Delta_{p,q} (\lambda_1, \lambda_2 , r, t  ) <  0 .$$
\end{theo}

\section{Proof of Theorem \ref{theo:1}}\label{sec:proof} 
The proof of Theorem \ref{theo:1} uses the self-duality of the contact process. For the sake of completeness we recall here this property. We start by introducing some notation. First of all, in the following we will not distinguish between the configuration $ \eta^ \lambda_t $ of the contact process at time $t$ and the associated subset of $\Z $ given by  $S( \eta^\lambda_t)  = \{ i \in \Z : \eta^\lambda_t (i ) = 1 \};$ that is, depending on the context, we will interpret $ \eta^\lambda_t $ as element of $X $ or as element of $ \PZ ( \Z) ,$ the set of all subsets of $\Z.$ Moreover, for any subset $ A \subset \Z , $ we write $ (\eta_t^{ \lambda, A} , t \geq 0 ) $ for the contact process starting from the initial configuration $ \eta^{\lambda, A}_0  (i) = 1 $ if and only if $i \in A.$ We observe that if $A$ is finite, then $(\eta^{\lambda, A}_t , t \geq 0 ) $ is just a pure jump Markov process, taking values in the set of finite subsets of $\Z. $ If $A = \{ i \}, $ then we write simply $ \eta_t^i $ for $\eta_t^{\{i\}} .$

The duality property of the contact process can be stated as follows. For any finite subset $A\in \PZ ( \Z) $ and any initial configuration $ \xi \in X, $ 
\begin{equation}\label{eq:duality}
 \P [ \, \eta^{\lambda , \xi}_t \cap A \neq \emptyset ]  = \P [ \eta_t^{\lambda , A } \cap \xi \neq \emptyset  ]. 
\end{equation} 
For more on duality see \cite{association} and \cite{harris1976}. 
 
 Our proof of Theorem \ref{theo:1} relies on the following result.

\begin{prop}\label{prop:1}
Let $\Lambda \in \PZ ( \Z )   $ be a finite subset of $\Z $  and assume that $ \xi (i )  , \xi' (i)  , i \in \Lambda  , $ are i.i.d.\ Bernoulli random variables with parameter $p,$ $q, $ respectively, for $0<  p < q< 1  ,$ and that $\xi (i)  = \xi' (i) = 0$ for all $i \in \Lambda^c . $  Then for all $i \in \Z,$
$$ \inf_{ \P \in \Pi   } \E ( | \eta_t^{\lambda , \xi} (0) -  \eta_t^{\lambda , \xi' } ( 0) | ) \\
=  \E \left[ (1-p)^{|\eta_t^{ \lambda ,0}  \cap \Lambda |} - (1-q)^{|\eta_t^{\lambda, 0} \cap \Lambda |}\right],
$$
where $\Pi$ denotes all possible couplings $\P $ of $\eta_t^{\lambda , \xi} (0) $ and $ \eta_t^{\lambda , \xi' } (0)  .$
\end{prop}

\begin{rem}
Notice that in our definition of sensitivity in \eqref{eq:dynrange} above, to make explicit the relationship with $p$ and $\Lambda , $ we wrote $   \eta_t^{\lambda , p, \Lambda  }  $ for $\eta_t^{\lambda , \xi } $  and $ \eta_t^{\lambda, q, \Lambda } $ for $\eta_t^{\lambda , \xi'} .$
\end{rem}

\begin{proof}
We take the maximal coupling of $ \xi $ and $\xi', $ that is, $ \xi (i)  \le \xi' (i) $ for all $i \in \Lambda  .$ Moreover, we use the canonical monotone coupling of  $ \eta^{\lambda , \xi} $ and $\eta^{\lambda , \xi' } ;$ that is, $ \eta_t^{\lambda , \xi }  \le \eta_t^{\lambda , \xi' } $ (in the sense of $ \eta_t^{\lambda , \xi } (i) \le \eta_t^{\lambda , \xi' } ( i )$ for all $i$) for all $ t \geq 0.$ Then
$$  | \eta_t^{\lambda , \xi} (0) -  \eta_t^{\lambda , \xi'}( 0) | = 1_{ \{ \eta_t^{\lambda,\xi } (0) = 0 , \eta_t^{\lambda , \xi'}( 0) = 1 \}}  .$$ 
We obtain by \eqref{eq:duality} and since $  \xi  \le \xi ' , $  
\begin{multline*}
 \P ( \eta_t^{\lambda , \xi } (0) = 0 , \eta_t^{\lambda ,\xi' }(0) = 1  ) = \P ( \xi  \cap  \eta^{ \lambda , 0 }_t  = \emptyset , \xi' \cap   \eta^{\lambda, 0}_t   \neq \emptyset   ) \\
 = \E [(1-p)^{|\eta^{\lambda, 0}_t   \cap \Lambda  | } - (1- q)^{|\eta_t^{ \lambda , 0 }   \cap \Lambda | }  ]. 
\end{multline*}
Therefore, 
$$
\inf_{ \P \in \Pi   } \E ( | \eta_t^{\lambda , \xi }(0) -  \eta_t^{\lambda, \xi' } ( 0) | )
 \le  \E [(1-p)^{|\eta^{\lambda, 0}_t   \cap \Lambda  | } - (1- q)^{|\eta_t^{ \lambda , 0}   \cap \Lambda | }  ].
$$

We now give a lower bound, following Lemma 6.1 of \cite{Galves2010}.  In the following, write for short $ \eta_t = \eta_t^{\lambda , \xi } $ and $ \eta_t' = \eta_t^{\lambda , \xi' }.$
We have
\begin{multline*}
\E |\eta_t (0) -  \eta_t'( 0) | = \P ( \eta_t  (0) \neq \eta'_t  (0)  ) \\
= \P ( \eta_t   (0) = 0,  \eta'_t  (0)  = 1 ) +  \P (  \eta_t (0)  = 1 ,  \eta'_t  (0)  = 0 ) .
\end{multline*}
This expression is minimized by the optimal coupling of $  \eta_t (0)$ and $ \eta'_t (0)$ given by 
$$  \P (  \eta'_t (0)= a,  \eta_t (0)  = a ) =  \P (  \eta'_t (0)= a ) \wedge \P( \eta_t (0)  = a ) $$
for $ a = 0, 1 ,$ 
$$ \P ( \eta_t  (0) = 0,  \eta'_t (0) = 1 )  = \P( \eta'_t  (0)  = 1 ) - \P (  \eta_t  (0) = 1,  \eta'_t(0)  = 1 ) $$
and 
$$\P ( \eta_t (0)= 1, \eta'_t (0)  = 0 ) = \P( \eta_t (0) = 1 ) - \P ( \eta_t (0)= 1, \eta'_t  (0) = 1 ) .$$
In this way, for any possible coupling, 
$$ \P ( \eta_t  (0) \neq \eta'_t (0)  )  \geq |  \P(  \eta_t  (0) = 1 ) - \P (  \eta'_t (0) = 1 ) | ,$$
which, due to \eqref{eq:duality}, equals
$$ |  \P(  \eta_t (0) = 1 ) - \P (  \eta'_t  (0) = 1 ) | = | \E [ 1 - ( 1-p)^{|\eta^{\lambda, 0}_t  \cap \Lambda |}] - \E [ 1 - ( 1-q)^{|\eta^{\lambda, 0}_t\cap \Lambda |}]| ,$$
implying the assertion.
\end{proof}

A second important ingredient for the proof of Theorem \ref{theo:1} is the  following monotone coupling construction of $ (\eta_t^{ \lambda_1 , A}  , t \geq 0) $ and  $(\eta_t^{\lambda_2 , A}, t \geq 0 )$ for $\lambda_1 < \lambda_2 ,$ where $A$ is some finite subset of $  \Z . $ 
We associate to each site $ i \in \Z$ five independent Poisson processes having jump times $ ( T_n^{ i , \dag })_n  $ with rate $ 1,  $ $ (T_n^{ i \to  i +1  })_n   $ with rate $ \lambda_1 , $ $ (T_n^{ i \to  i -1  })_n   $ with rate $ \lambda_1 , $ $ (S_n^{ i \to  i +1  })_n   $ with rate $ \lambda_2 - \lambda_1  $ and finally  $ (S_n^{ i \to  i -1  })_n   $ with rate $ \lambda_2 - \lambda_1 . $ We assume that the processes attached to different sites are all independent. We then construct $ \eta^{ \lambda_1 , A}  $ and  $\eta^{\lambda_2 , 0}$ in the following way. Firstly, both processes start from the same initial configuration $A$ at time $0.$ Then we update the configurations according to the following rules.
 
\begin{itemize}
\item
Every time that $T_n^{i, \dag } $ rings, both processes simultaneously upgrade the value  at site $i$ to $0.$
\item
Every time that $ T_n^{ i \to  i +1  } $ rings, both processes simultaneously try to upgrade the position at site $i+1$ to $1, $ provided that at site $i$ or at site $ i+1$ there is a symbol $1.$ 
\item
Every time that $ T_n^{ i \to  i - 1  } $ rings, both processes simultaneously try to upgrade the position at site $i- 1$ to $1, $ provided that at site $i$ or at site $ i-1$ there is a symbol $1.$  
\item
Every time that $ S_n^{ i \to  i +1  } $ rings, only the process $\eta^{\lambda_2 , A}$ tries to upgrade the position at site $i+1$ to $1, $ provided that at site $i$ or at site $ i+1$ there is a symbol $1.$ 
\item
Every time that $ S_n^{ i \to  i -1  } $ rings, only the process $\eta^{\lambda_2 , A}$ tries to upgrade the position at site $i-1$ to $1, $ provided that at site $i$ or at site $ i-1$ there is a symbol $1.$ 
\end{itemize}

With this construction, we obtain the following proposition. 

\begin{prop}\label{prop:coup}
For the above coupled construction of $ ( \eta^{\lambda_1 , A }_t , t \geq 0 ) $ and $ ( \eta^{\lambda_2 , A }_t , t \geq 0 ) ,$ the following holds. 
\begin{itemize}
\item
$ \eta^{ \lambda_1, A }_t \le \eta^{\lambda_2, A }_t$ for all $ t \geq 0 .$
\item
For any fixed site $ r \in \Z, $ $ \{ \eta^{\lambda_1, A}_t ( r) = 1 \} $ is conditionally independent of $\{  \eta^{\lambda_2, A}_t ( -r) = 1\} ,$ conditionally on $  \{ \eta^{\lambda_1, A}_t ( -r) = 0 \}.$ 
\end{itemize}
\end{prop}

Finally, we will rely on the following well-known result. We define 
$$ \varrho ( \lambda ) = \P (  \eta_t^{ \lambda, 0 } \mbox{ survives forever } ) .$$ 

\begin{theo}[Theorems 1.6 and  2.28 in Chapter VI of \cite{Liggett1985}]\label{theo:liggett} 
The following properties hold.
\begin{itemize}
\item 
$ \varrho ( \lambda ) = 0 $ for $\lambda \le \lambda_c, $ and $\varrho (\lambda ) > 0 $ for all $ \lambda  > \lambda_c .$
\item
The function $ \varrho ( \lambda ) $ is continuous and non-decreasing   in $\lambda  ,$ and $\varrho ( \lambda  ) \uparrow 1$ as $\lambda \uparrow \infty .$ 
\item
For all $\lambda > \lambda_c , $ there exists a unique probability measure $\nu_{\lambda} $ on $ \PZ ( \Z ) $ such that for any cylinder function $f, $
$$ \lim_{t \to \infty } \E ( f ( \eta_t^{ \lambda, 0} ) ) = ( 1 - \varrho ( \lambda  ) ) f( 0)  + \varrho ( \lambda) \int f d \nu_{ \lambda} ,$$
where $ 0$ denotes the configuration $ \xi \equiv 0 $ ``all-zero''.
\item 
The measure $ \nu_{\lambda } $ has exponentially decaying correlations, that is, there exist $ C , \varepsilon  > 0 $ with the following property. For all cylinder functions $f_1$ and $f_2 $ such that $ f_1 ( B ) $ depends only on $ B \cap R_1 $ and $f_2 (B) $ only on $ B \cap R_2,$  for some fixed $ R_1, R_2$ which are finite subsets of $  \Z  , $ 
\begin{equation}\label{eq:expdecay}
| \nu_{\lambda} ( fg) - (\nu_{\lambda} f) (\nu_{\lambda}  g)|  \le C e^{ - \varepsilon \, dist (R_1 , R_2) } .
\end{equation} 
\end{itemize}
\end{theo}
The monotonicity of $\varrho ( \lambda )$ follows from the  construction presented in Proposition \ref{prop:coup} above. For the remaining results, we refer the interested reader to \cite{Liggett1985} for a proof and references. 

We are now able to prove Theorem \ref{theo:1}.

\begin{proof}[Proof of Theorem \ref{theo:1}]

{\bf Step 1.} We start with the case $\lambda_1 < \lambda_2 < \lambda_c.$

Define $ f(x) = (1-p)^x - (1-q) ^x , $ for any $x \in \N.$ We rely on the coupled construction of $ \eta^{ \lambda_1, 0 }_t $ and $ \eta^{\lambda_2, 0 }_t$ of Proposition \ref{prop:coup} above. By Proposition \ref{prop:1}, we may therefore write
$$\Delta_{p,q} (\lambda_1, \lambda_2, r, t  ) = \E_{\lambda_1, \lambda_2 }  \left( f(  | \eta^{\lambda_2, 0}_t  \cap \Lambda_r |) - f( | \eta^{\lambda_1 , 0}_t \cap \Lambda_r |) \right) ,  
$$
where $ \E_{\lambda_1, \lambda_2 }  $ denotes the expectation with respect to this monotone coupling of $\eta^{ \lambda_1 , 0}_t $ and $ \eta^{\lambda_2 , 0}_t.$ Then, by the monotonicity and since $f(0)=0,$ 
\begin{multline*}
\Delta_{p,q} (\lambda_1, \lambda_2, r, t  )  = [ f(2) - f(1) ] \P_{\lambda_1, \lambda_2 } [ | \eta_t^{\lambda_2 , 0 } \cap \Lambda_r | = 2 , | \eta_t^{\lambda_1 , 0 } \cap \Lambda_r |= 1 ] \\
+ f(2) \P_{\lambda_1, \lambda_2 } [ | \eta_t^{\lambda_2 , 0 } \cap \Lambda_r | = 2 , | \eta_t^{\lambda_1 , 0 } \cap \Lambda_r |= 0  ]\\
+ f(1)   \P_{\lambda_1, \lambda_2 } [ | \eta_t^{\lambda_2 , 0 } \cap \Lambda_r | = 1 , | \eta_t^{\lambda_1 , 0 } \cap \Lambda_r |= 0  ].
\end{multline*}
Notice that $ f(2) - f(1) = (q-p) ( 1 - p -q ) < 0 ,$ since by assumption, $ 1/2 < p \le q .$ 
Therefore, 
\begin{multline*}
\Delta_{p,q} (\lambda_1, \lambda_2,r, t  )\geq [ f(2) - f(1) ] \P_{\lambda_1, \lambda_2 } [ | \eta_t^{\lambda_2 , 0 } \cap \Lambda_r | = 2 , | \eta_t^{\lambda_1 , 0 } \cap \Lambda_r |= 1 ] \\
+ f(2) \P_{\lambda_1, \lambda_2 } [ | \eta_t^{\lambda_2 , 0 } \cap \Lambda_r | \neq  0 , | \eta_t^{\lambda_1 , 0 } \cap \Lambda_r |= 0  ] .
\end{multline*}
By symmetry, 
\begin{multline*}
 \P_{\lambda_1, \lambda_2 } [ | \eta_t^{\lambda_2 , 0 } \cap \Lambda_r | = 2 , | \eta_t^{\lambda_1 , 0 } \cap \Lambda_r |= 1 ] \\
= 
2 \P_{\lambda_1, \lambda_2 } [ \eta^{\lambda_2 , 0 }_t (\pm r ) = 1 ,   \eta^{\lambda_1 , 0 }_t ( -r) = 0,   \eta^{\lambda_1 , 0 }_t ( r)= 1 ],
\end{multline*}
and 
\begin{multline*}
\P_{\lambda_1, \lambda_2 } [ | \eta_t^{\lambda_2 , 0 } \cap \Lambda_r | \neq  0 , | \eta_t^{\lambda_1 , 0 } \cap \Lambda_r |= 0  ]\\
\le 2 \P_{\lambda_1, \lambda_2 } [ \eta^{\lambda_2 , 0 }_t (- r ) = 1 ,   \eta^{\lambda_1 , 0 }_t ( \pm r) = 0 ] .
\end{multline*}

We put $A_t := \{ \eta^{\lambda_2 , 0 }_t (- r ) = 1 ,   \eta^{\lambda_1 , 0 }_t (-r) = 0 \} . $ By monotonicity we obtain
\begin{multline*}
\Delta_{p,q} (\lambda_1, \lambda_2,r, t  ) \geq 2 \P_{\lambda_1, \lambda_2 } ( A_t) \Big(  [ f(2) - f(1) ] \P_{\lambda_1, \lambda_2 } ( \eta^{\lambda_1, 0}_t ( r) = 1 | A_t ) \\
+ f(2) \P_{\lambda_1, \lambda_2 } ( \eta^{ \lambda_1, 0 }_t  (r) = 0 | A_t )\Big)  .
\end{multline*}
We now use the fact that the event $ \{ \eta^{\lambda_1, 0}_t ( r) = 1 \} $ is conditionally independent of $\{  \eta^{\lambda_2, 0}_t ( -r) = 1\} ,$ conditionally on $  \{ \eta^{\lambda_1, 0}_t ( -r) = 0 \}.$ Thus 
$$ \P_{\lambda_1, \lambda_2 } ( \eta^{\lambda_1, 0}_t ( r) = 1 | A_t ) = \P ( \eta^{\lambda_1, 0}_t ( r) = 1 | \eta^{\lambda_1 , 0 }_t (-r) = 0  ) .$$
Analogously,  
$$ \P_{\lambda_1, \lambda_2 } (\eta^{ \lambda_1, 0 }_t  (r) = 0| A_t ) = \P ( \eta^{\lambda_1, 0}_t ( r) = 0 | \eta^{\lambda_1 , 0 }_t (-r) = 0  ) ,$$
implying that 
\begin{multline*}
\Delta_{p,q} (\lambda_1, \lambda_2,r, t  ) \geq  2 \P_{\lambda_1, \lambda_2 } ( A_t) \Big(  [ f(2) - f(1) ] \P ( \eta^{\lambda_1, 0}_t ( r) = 1 | \eta^{\lambda_1 , 0 }_t (-r) = 0  ) \\
+ f( 2) \P ( \eta^{\lambda_1, 0}_t ( r) = 0 | \eta^{\lambda_1 , 0 }_t (-r) = 0  ) \Big).
\end{multline*}
Now, if  $ \lambda_1 < \lambda_c, $ we have that 
$$ \P ( \eta^{\lambda_1, 0}_t ( r) = 1;  \eta^{\lambda_1 , 0 }_t (-r) = 0) \to 0 $$
and 
$$ \P ( \eta^{\lambda_1, 0}_t ( r) = 0 ;  \eta^{\lambda_1 , 0 }_t (-r) = 0  ) \to 1$$
as $ t \to \infty .$ Therefore,  there exists $ t_* $ depending on $ f(2) , f(1) $ and on $r,$ such that for all $t \geq t^*, $ 
\begin{multline*}
  [ f(2) - f(1) ] \P ( \eta^{\lambda_1, 0}_t ( r) = 1 | \eta^{\lambda_1 , 0 }_t (-r) = 0  ) 
\\
+ f( 2) \P ( \eta^{\lambda_1, 0}_t ( r) = 0 | \eta^{\lambda_1 , 0 }_t (-r) = 0  ) > 0 .
\end{multline*}
Since $ \P_{\lambda_1, \lambda_2 } ( A_t) > 0 $ for all finite $t,$ this implies the first assertion in the subcritical case $ \lambda_1 < \lambda_2 < \lambda_c .$ 

Let us now consider values of $ \lambda_1 $ which are slightly above $\lambda_c.$ 
Relying on Theorem \ref{theo:liggett}, we have 
\begin{multline}\label{eq:first}
\lim_{ t \to \infty} \P ( \eta^{\lambda_1, 0}_t ( r) = 1 | \eta^{\lambda_1, 0}_t ( -r) = 0 )  \\
=   \frac{ \varrho ( \lambda_1 )  \nu_{\lambda_1} ( \{ B : B \cap \{ - r \} = 0, B \cap \{ r\} = 1 \}  ) }{ ( 1 - \varrho ( \lambda_1) ) + \varrho ( \lambda_1) \nu_{\lambda_1} ( \{ B : B \cap \{ - r \} = 0 \} ) } 
\end{multline}
and 
\begin{multline}\label{eq:second}
\lim_{ t \to \infty}  \P ( \eta^{ \lambda_1, 0 }_t  (r) = 0 | \eta^{\lambda_1, 0}_t (-r) = 0  ) \\
= 
 \frac{( 1 - \varrho ( \lambda_1) ) + \varrho ( \lambda_1) \nu_{\lambda_1} ( \{ B : B \cap \{ \pm  r \} = 0\}) }{( 1 - \varrho ( \lambda_1) ) + \varrho ( \lambda_1) \nu_{\lambda_1} ( \{ B : B \cap \{ -  r \} = 0 )\} ) } .
\end{multline} 
By Theorem \ref{theo:liggett},  $ \varrho ( \lambda) \downarrow 0 $ as $ \lambda \downarrow \lambda_c .$ 
Observe that 
$$\nu_{\lambda} ( \{ B : B \cap \{   r \} = 1 )\} )  =  \varrho ( \lambda ). $$ 
Using \eqref{eq:expdecay} we deduce that 
\begin{equation}\label{eq:third}
\lim_{r \to \infty } \nu_{\lambda_1} ( \{ B : B \cap \{ - r \} = 0, B \cap \{ r\} = 1 \}  ) = \varrho ( \lambda_1) (1 - \varrho ( \lambda_1)) ,
\end{equation}
with analogous formulas for the other terms appearing in \eqref{eq:first} and \eqref{eq:second} above.  

Therefore, fix some $ \varepsilon > 0 $ and choose $ \lambda_1 (p ) > \lambda_c $ sufficiently close to $ \lambda_c $ such that 
\begin{multline*}
f(2) \frac{( 1 - \varrho ( \lambda_1) ) + \varrho ( \lambda_1) ( 1 -\varrho (\lambda_1))^2  }{( 1 - \varrho ( \lambda_1) ) + \varrho ( \lambda_1) (1- \varrho (\lambda_1))    } \\
-   | f(2) - f(1) |  \frac{ \varrho ( \lambda_1 )  \varrho ( \lambda_1) ( 1 - \varrho  ( \lambda_1) ) }{ ( 1 - \varrho ( \lambda_1) ) + \varrho ( \lambda_1) ( 1 - \varrho  (\lambda_1))  }   
\geq  \varepsilon 
\end{multline*}
for all $ \lambda_c \le \lambda_1 \le  \lambda_1 (p)  .$ 

Now, fix any $ \lambda_1 \in [ \lambda_c,  \lambda_1 (p ) ].$ Thanks to the above convergence results \eqref{eq:first}--\eqref{eq:third}, it is possible to choose first $r^* $ and then $t^*  $ such that for all $ r \geq r^* $ and $t \geq t^*, $   
\begin{multline*}
 f(2) \P ( \eta^{ \lambda_1, 0 }_t  (r) = 0 | \eta^{\lambda_1, 0}_t (-r) = 0  )  - \\
 | f(2) - f(1) | \P ( \eta^{\lambda_1, 0}_t ( r) = 1 | \eta^{\lambda_1, 0}_t ( -r) = 0 ) 
\geq  
 \frac{\varepsilon  }{2},
\end{multline*}
implying that 
$$ \Delta_{p,q} (\lambda_1, \lambda_2, r, t  ) \geq   \varepsilon \;    \P_{\lambda_1, \lambda_2 } ( A_t) > 0 .$$
This concludes the proof of  the first assertion.

{\bf Step 2.} We finally consider the case where $\lambda_1 $ is sufficiently larger than $\lambda_c.$ Let $Q$ be the monotone coupling between $ \nu_{\lambda_1}$ and $\nu_{\lambda_2 }$ induced by the construction of Proposition \ref{prop:coup}. Using this coupling and the fact that $ f(0 ) = 0,$ we obtain thanks to Theorem \ref{theo:liggett} that 
\begin{multline}
\lim_{t \to \infty } \Delta_{p,q} (\lambda_1, \lambda_2,r, t  ) \\
= \varrho ( \lambda_2)  \int_{\PZ ( \Z) }  f( | B \cap \Lambda |)  \nu_{\lambda_2} ( dB) - \varrho ( \lambda_1) 
 \int_{\PZ ( \Z) }  f( | B \cap \Lambda |)  \nu_{\lambda_1} ( dB)\\
= \varrho ( \lambda_1 )   \int \int 
\left( f( |B_2  \cap \Lambda_r |) - f ( 
|B_1  \cap \Lambda_r | ) \right)    Q ( d B_1, d B_2 ) \\
+ ( \varrho ( \lambda_2 ) - \varrho ( \lambda_1 ) )  \int_{\PZ ( \Z) }  f( | B \cap \Lambda |)  \nu_{\lambda_2} ( dB) \\
=: T_1 (r) + T_2 (r)  .
\end{multline}
We want to show that this expression is negative for sufficiently large values of $ \lambda_1 $ and  $r.$
We put $ \varepsilon := 2 - p - q  . $ Since by assumption $ q > p \geq \frac23, $ we have $ 2 ( 1 - \varepsilon ) > \varepsilon $ (this will be important in \eqref{eq:imp} below).

Then $ f(2) - f(1) = -( 1 - \varepsilon)  f(1) $ and $ f(2) = \varepsilon  f(1) .$  Writing for short 
$$ Q (r, n, m ) := Q (\{ (B_1, B_2 ) : |B_1 \cap \Lambda_r |  = n, |B_2 \cap \Lambda_r  | = m \} ) , $$ 
it is clear that 
\begin{eqnarray*}
 T_1 (r)  &=& \varrho ( \lambda_1) \left[ (f(2) - f(1) ) Q(r, 1, 2) + f(2) Q(r, 0, 2 ) + f(1) Q(r, 0, 1 )\right] \\
 &=&  \varrho ( \lambda_1) f(1) \left[ -(1- \varepsilon) Q(r, 1, 2 ) + \varepsilon Q(r, 0, 2 ) + Q(r, 0, 1 ) \right] 
.  
\end{eqnarray*}
Applying the last item of Theorem \ref{theo:liggett}, we have that 
\begin{multline*}
\lim_{r \to \infty } T_1 (r) =  \varrho ( \lambda_1) f(1) \Big(  -2 (1- \varepsilon) (\varrho ( \lambda_2) - \varrho ( \lambda_1 ) )\varrho ( \lambda_1)  \\
+ \varepsilon  (\varrho ( \lambda_2) - \varrho ( \lambda_1 ) )^2 +
2  (\varrho ( \lambda_2) - \varrho ( \lambda_1 ) ) (1- \varrho ( \lambda_2) ) \Big) \\
= \varrho ( \lambda_1) f(1)   (\varrho ( \lambda_2) - \varrho ( \lambda_1 ) ) \Big(  -2 (1- \varepsilon) \varrho ( \lambda_1) + \varepsilon  (\varrho ( \lambda_2) - \varrho ( \lambda_1 ) ) \\
 +
2  (1- \varrho ( \lambda_2) ) \Big) .
\end{multline*}
 
Moreover, 
$$  \lim_{r \to \infty } T_2 (r)  = ( \varrho ( \lambda_2 ) - \varrho ( \lambda_1 ) )  \left[ 2 f(1) \varrho ( \lambda_2 ) (1- \varrho ( \lambda_2) ) + f(2) \varrho (\lambda_2)^2  \right] .$$ 
Putting these results together, we conclude that 
\begin{multline}
\lim_{r \to \infty } \lim_{t \to \infty } \Delta_{p,q} (\lambda_1, \lambda_2,r, t  ) = ( \varrho ( \lambda_2 ) - \varrho ( \lambda_1 ) ) f(1) \\
\Big\{  -2 (1- \varepsilon) (\varrho ( \lambda_1))^2 + \varepsilon  \varrho ( \lambda_1)  (\varrho ( \lambda_2) - \varrho ( \lambda_1 ) )  +
2  (1- \varrho ( \lambda_2) ) \varrho ( \lambda_1)  \\
+ 2  \varrho ( \lambda_2 ) (1- \varrho ( \lambda_2) ) + \varepsilon \varrho (\lambda_2)^2  
\Big\} .
\end{multline}

Since $ 2 ( 1 - \varepsilon ) > \varepsilon, $ it is possible to choose $ \delta^*$ such that for all $ \delta \le \delta^* , $
\begin{equation}\label{eq:imp}
2 ( 1 - \varepsilon) (1- \delta)^2 - \varepsilon ( 1- \delta) \delta - 4 \delta - \varepsilon ( 1 - \delta)^2 \geq \kappa > 0 ,
\end{equation}
for some (sufficiently small) $ \kappa > 0 .$  Recall that $ \varepsilon  = \varepsilon ( p) .$ Since $ \lim_{\lambda \uparrow \infty } \varrho ( \lambda )  = 1,$ 
we may choose $ \lambda_2 (p ) $ sufficiently large such that $\varrho ( \lambda_1 ) \geq 1 - \delta^* $ for all $ \lambda_1 \geq \lambda_2 (p) .$ 
As a consequence, for all $ \lambda_2 \geq  \lambda_1 \geq \lambda_2 (p) ,$
\begin{equation}
\lim_{r \to \infty} \lim_{t \to \infty }  \Delta_{p,q} (\lambda_1, \lambda_2,r, t  ) \geq \kappa > 0 ,
\end{equation}
which implies the assertion.
\end{proof}

\section{Final discussion}\label{sec:finaldiscussion}

In the present article, we have proved that for the contact process the information transmission -- as defined in \eqref{eq:dynrange} and \eqref{eq:sensitivitydiff} -- is maximized at a value of the control parameter $\lambda $ which is strictly larger than the critical value $\lambda_c .$ Similar issues have been discussed by many other authors, using different measures of information transmission and considering different models. In the present section, we give an overview of these results and compare our findings to the ones already established in the literature. 

A commonly used measure to quantify  information transfer is a recent  information theoretic measure introduced by \cite{schreiber}, the so-called {\it transfer entropy}. This {\it transfer entropy} quantifies ``the statistical coherence between systems evolving in time''  (cf. \cite{schreiber}). It is ``able to distinguish driving and responding elements and to detect asymmetry in the coupling of subsystems'' (cf. \cite{schreiber}).  An important point is that this quantity measures ``to which extent the individual components contribute to information production and at what rate they exchange information among each other'', when an external perturbation is absent (cf. \cite{schreiber}). 

In the case of ferromagnetic Ising models, \cite{Barnett} show numerically that this transfer entropy is maximized in the disordered phase, that is, in the region where only one invariant measure exists and which would correspond to the subcritical regime for the contact process. This result is confirmed by the findings of \cite{entropy2}. \cite{Barnett} argue that their result could be related to a subtle interplay between sites within and out  the boundaries of same spin domains  whose probability distributions are a function of the temperature. 
On the other hand, \cite{entropy} consider a Susceptible-Infected-Susceptible (SIS) epidemic model on a homogeneous network and provide simulations showing that the transfer entropy is maximized in the supercritical regime, confirming our result. They argue that ``once the disease dynamics reach criticality, we observe strong effects of one individual on a connected neighbor (measured by the transfer entropy). However, as the dynamics become supercritical, the target neighbor becomes more strongly bound to all of its neighbors collectively, and it becomes more difficult to predict its dynamics based on a single source neighbor alone; as such, the transfer entropy begins to decrease'' (see \cite{entropy}). These results are very close to the ones we have found in the present paper for the contact process. 

Our paper presents two main differences with respect to the above cited ones. First of all, to the best of our knowledge, our result is the first one available using analytical methods instead of numerical simulations. The second difference is that instead of relying on the transfer entropy, we  measure how much a system discriminates between different external inputs to which the system is initially exposed.  To do so, we have introduced the notion of  {\it sensitivity} of the model with respect to the initial signal, given by the total variation distance between the two associated processes at a given single site, at some fixed time. 

Let us briefly comment on this choice. Two points of view are commonly adopted to describe the influence of external stimuli in neuronal systems. On the one hand, one might think of external stimuli which are permanently influencing the system, acting as external field. This is the point of view adopted by \cite{KinouchiCopelli}. The second approach is to think of an initial configuration, coming  from another region of the brain, which is exposed as in initial stimulus to the region one is interested in, and to see how this initial stimulus is propagated by the system. This is the point of view we have adopted in the present paper, leading to our definition of {\it sensitivity}.  

Although our measure of information transfer is different from those used by \cite{Barnett}, \cite{entropy} and also \cite{KinouchiCopelli}, the fact that it is maximized in the supercritical region, that is, in the ordered phase where several invariant measures coexist, is due to the very nature of the system we consider (in the very same way as what was observed in \cite{entropy}). This is due to the fact that both the SIS epidemics model as well as the contact process describe the evolution and the spread of an epidemics. More precisely, our result can be related to the specific features of the stationary states of our model where $\varrho ( \lambda) $  is strictly larger than zero only for $ \lambda > \lambda_c $ (cf. Theorem \ref{theo:liggett}). In other words, to convey, at large times, a non  trivial  amount of information, $\lambda$ has to be larger than $ \lambda_c.$  On the contrary to these results, in the case of the 2d Ising model, \cite{Barnett} observe a peak on the disordered site, that is, in the region, where only one invariant measure exists. This result is characteristic of the very nature of the Ising model and shows that in terms of its information theoretic structure, the Ising model displays different features than the contact process or any other model of epidemics spread.

\section*{Acknowledgements}
Many thanks to Errico Presutti and  Antonio Carlos Roque da Silva Filho for stimulating discussions about this subject. We also thank two anonymous referees for helpful comments and suggestions.
We thank the Gran Sasso Science Institute (GSSI) for hospitality and support. 
This research has been conducted as part of the project Labex MME-DII (ANR11-LBX-0023-01), USP project {\em Mathematics, computation, language
and the brain} and FAPESP project {\em Research, Innovation and
Dissemination Center for Neuromathematics} (grant 2013/07699-0). AG is partially supported  by CNPq fellowship (grant 311 719/2016-3.)

\nocite{harris1978}
\nocite{KinouchiCopelli}
\bibliography{biblio}{}
\bibliographystyle{plain}

\end{document}